\documentclass[reqno,11pt]{amsart}

\calclayout

\usepackage[utf8]{inputenc}
\usepackage[T1]{fontenc}

\usepackage{amssymb, amsmath, amsthm}
\usepackage{esint}
\usepackage{tikz-cd}
\usepackage{graphicx}
\usepackage{mathtools}
\mathtoolsset{showonlyrefs}
\usepackage{xspace}
\theoremstyle{definition}
\newtheorem{defi}{Definition}[section]

\newtheorem{remark}[defi]{Remark}

\theoremstyle{plain}
\newtheorem{theorem}[defi]{Theorem}
 \newtheorem{prop}[defi]{Proposition}
\newtheorem{lemma}[defi]{Lemma}

\usepackage{marginnote}

\usepackage{color}

\numberwithin{equation}{section}

\usepackage[pdftitle={Weyl's law for the Steklov problem on surfaces with rough
  boundary},
  pdfauthor={Mikhail Karpukhin, Jean Lagacé and Iosif Polterovich}, 
pdfborder={0 0 0}
]{hyperref}

\makeatletter

\let\TagsLeftOn\tagsleft@true
\let\TagsLeftOff\tagsleft@false
\makeatother

\newcommand{\D}{\mathbb{D}}
\newcommand{\R}{\mathbb R}
\newcommand{\N}{\mathbb N}

\newcommand{\de}{\, \mathrm{d}}
\newcommand{\del}{\partial}

\newcommand{\Vol}{\operatorname{Vol}}

\newcommand{\CC}{\mathcal{C}}

\newcommand{\CH}{\mathcal{H}}

\newcommand{\CK}{\mathcal{K}}

\newcommand{\CM}{\mathcal{M}}

\newcommand{\CW}{\mathcal{W}}
\newcommand{\CX}{\mathcal{X}}
\newcommand{\CY}{\mathcal{Y}}

\newcommand{\rd}{\mathrm d}

\newcommand {\RL}{\mathrm L}
\newcommand {\RA}{\mathrm A}

\newcommand{\RW}{\mathrm W}
\newcommand{\RC}{\mathrm C}

\newcommand{\abs}[1]{\left\lvert #1 \right\rvert}

\newcommand{\set}[1]{\left\{ #1 \right\}}

\newcommand{\norm}[1]{\left\| #1 \right\|}
\newcommand{\bo}\boldsymbol{}

\newcommand{\smallo}[2][]{o_{#1}\left( #2 \right)}
\newcommand{\bone}{\bo{1}}

\DeclareMathOperator{\dist}{dist}

\renewcommand{\le}{\leqslant}
\renewcommand{\ge}{\geqslant}

\newcommand{\eps}{\varepsilon}
\renewcommand{\phi}{\varphi}

\DeclareMathOperator{\Per}{Per}

\begin{document}

\title{Weyl's law for the Steklov problem on surfaces with rough boundary}

\author[Mikhail Karpukhin]{Mikhail Karpukhin}
\address{Mathematics 253-37, Caltech, Pasadena, CA 91125, USA
}
\email{mikhailk@caltech.edu}

\author[Jean Lagac\'e]{Jean Lagac\'e}

\address{Department of Mathematics, King's College London, The Strand, London,
WC2R 2LS, United Kingdom}

\email{jean.lagace@kcl.ac.uk}

\author[Iosif Polterovich]{Iosif Polterovich}

\address{D\'epartement de math\'ematiques et de statistique, Universit\'e de
  Montr\'eal, CP 6128, succ. Centre-ville, Montr\'eal, Qu\'ebec, H3C 3J7,
Canada}

\email{iossif@dms.umontreal.ca}

\begin{abstract}
The validity of Weyl's law for the Steklov problem on domains with Lipschitz boundaries is a well-known open question in spectral geometry.
We answer this question in two dimensions  and show that   Weyl's law holds for an  even  larger class of
surfaces  with rough boundaries. This class includes  domains with  interior cusps as well as  ``slow'' exterior cusps. Moreover, 
the condition on the speed of exterior cusps cannot be improved, which makes our result in a sense optimal. 
The  proof is based on the methods of Suslina and Agranovich 
combined with some
  observations about the boundary behaviour of conformal mappings.
\end{abstract}
\maketitle

\section{Introduction and main results}
\subsection{Asymptotics of Steklov eigenvalues}
Let $\Omega$ be a bounded domain in a smooth complete Riemannian manifold $(\CM,g)$ of dimension $d \ge 2$. Consider the Steklov eigenvalue problem
\begin{equation}
  \label{prob:Steklov}
  \begin{cases}
    \Delta u = 0 & \text{in } \Omega; \\
    \del_\nu u = \sigma u & \text{on } \del \Omega,
  \end{cases}
\end{equation}
where $\Delta$ is the Laplace-Beltrami operator on $\CM$ associated with the
Riemannian metric $g$ and $\del_\nu$ is the outward normal derivative. Under
some regularity
conditions on  the boundary, for instance, if $\del \Omega$ is Lipschitz
\cite{arendtmazzeo},  the spectrum is discrete and forms a sequence
accumulating only at infinity:
\begin{equation}
  0 = \sigma_0(\Omega) \le \sigma_1(\Omega) \le \sigma_2(\Omega) \le \dotso
  \nearrow \infty.
\end{equation}
We will discuss weaker conditions under which the spectrum is discrete later on. To study eigenvalue asymptotics, it is convenient to introduce the
eigenvalue counting function
\begin{equation}
  N(\sigma) := \# \set{j \in \N : \sigma_j(\Omega) < \sigma}.
\end{equation}
If  $\del \Omega$ is piecewise
$\RC^1$, it is known \cite{Agranovich}  that the counting function
satisfies the Weyl asymptotics
\begin{equation}
  \label{eq:weyllaw}
  N(\sigma) = \frac{\omega_{d-1}}{(2\pi)^{d-1}} \Vol_{d-1}(\del \Omega)\sigma^{d-1} + \smallo{\sigma^{d-1}},
\end{equation}
where $\omega_d$ is the volume of the $d$-dimensional unit ball.  We refer also to \cite{Sandgren, Shamma, Suslina1999, AgranovichAmosov} for earlier 
results on this topic, as well as to \cite{rozenbljumOriginal, Edward, GPPS, GKLP} for improvements of the error estimate under stronger regularity assumptions.
Extending the asymptotic formula \eqref{eq:weyllaw} to domains  with Lipschitz
boundaries is a well known open problem, see e.g. \cite{gpsurvey, RozenShar, Suslina2021,
GKLP}, to which we provide an answer in two dimensions.
\begin{theorem}
  \label{thm:main}
  Let $\Omega$ be a bounded domain with Lipschitz boundary in a smooth complete Riemannian manifold of dimension two. Then its Steklov eigenvalues satisfy the asymptotics \eqref{eq:weyllaw}.
\end{theorem}
In fact, we prove that \eqref{eq:weyllaw} holds for domains satisfying weaker regularity
conditions defined via the boundary behaviour of conformal maps, see Section
\ref{sec:conformalregularity}. We give examples of domains satisfying those
conditions in Section~\ref{sec:examples}, they include the so-called chord-arc domains,  as well as
domains with  inward  and   ``slow'' outward cusps, see Proposition
\ref{prop:cusps}.

The proof of Theorem \ref{thm:main} relies on the variational characterisation
of Steklov eigenvalues.  While this characterisation is standard for Lipschitz domains, 
certain subtleties arise for domains with less regular boundary which we 
clarify in Section~\ref{sec:variational}. We use a conformal map to obtain an isospectral weighted
Steklov problem on a surface with smooth boundary.  The isospectrality follows
from the  equivalence of the corresponding variational characterisations;  for Lipschitz domains,
it can  be deduced almost directly from the analogous results for the Neumann problem \cite{gu}. Finally, we use the
methods developed in \cite{birmansolomjaksobolev,Suslina1999}, see also \cite{Agranovich},  in order to
obtain spectral asymptotics for these weighted Steklov problems.  
For Lipschitz domains we could use the results from \cite{Suslina1999,Agranovich} in a straightforward way; we extend these 
techniques  to allow for more singular weights corresponding to less regular boundaries.
\subsection{Conformal regularity}
\label{sec:conformalregularity}
As was mentioned above, our first goal  is to reduce the Steklov problem on a surface with rough boundary to
a weighted Steklov problem on a surface with smooth boundary, via a conformal map.
Slightly abusing terminology, we refer to domains in
two-dimensional Riemannian manifolds whose boundary is a finite collection of
disjoint closed simple curves as \emph{surfaces with boundary}. 
We say that two surfaces $\Omega_1$ and $\Omega_2$ with (potentially empty) boundary are
\emph{conformally equivalent}, or \emph{in the same conformal class} if there
exists $\phi : \Omega_1 \to \Omega_2$ a conformal diffeomorphism of their
interior extending to a homeomorphism of their boundary.  This defines an equivalence relation on surfaces with boundary, and it is clear that
every conformal class consists of surfaces with the same topological type, i.e.
same orientability, genus and number of boundary components.

The uniformisation theorems are concerned with finding a canonical representative
in every conformal class $\CC$. These canonical representatives are \emph{circle
domains}, which are the complement of $b$ geodesic disks in a closed surface
$M$
endowed with a metric of constant curvature. It follows from the uniformisation
theorems \cite{Haas1984,Maskit1989}
that there is a circle domain in every orientable conformal class with finite
topology; this  result was extended to the non-orientable setting in
\cite[pp. 11--12]{KarpukhinStern2021}. 
We shall denote this canonical representative $(\Omega_\CC,g_\CC)$.

Many boundary regularity results in the litterature are proven for conformal maps from the disk to simply
connected surfaces with boundary. It follows from \cite[p.24]{bellkrantz} that
any such result for maps from the disk is also valid for maps from an annulus
into a doubly connected surface with boundary,
by alternately filling the boundary components of the target with a disk and conjugating with
inversions. As observed in \cite[Remark 2.2]{KarpukhinStern2021}, this allows
one to extend the regularity theory to conformal maps from arbitrary circle domains with finite
topology. Indeed, in that situation the restriction of the conformal map to a
neighbourhood of one boundary component in the circle domain is a map from an
annulus into a doubly connected surface with boundary. In particular, 
Carathéodory's Theorem tells us that any conformal diffeomorphism of the
interiors $\phi :
\Omega_\CC \to \Omega$ extends to a
homeomorphism of the boundary \cite[Theorem 2.6]{Pommerenkebook}. 
\begin{defi}
    \label{def:conformal factors}
    Let $\CC$ be a conformal class and $\phi : \Omega_\CC \to \Omega$ be a
    conformal diffeomorphism. We define when they exist,
    \begin{equation}
        \beta := \abs{\rd \phi \big|_{\del \Omega_\CC}} \qquad \text{and} \qquad
        \eta := \abs{\rd \phi}^2.
    \end{equation}
    We call $\beta $ the \emph{boundary conformal factor} and $\eta$ the
    \emph{interior conformal factor}.
\end{defi}
The interior conformal factor $\eta \in \RL^1(\Omega_\CC)$ and the Riemannian
volume measure $\rd v_g$ on $\Omega$ is the pushforward measure $\phi_*(\eta \rd
v_{g_\CC})$. 
If a
surface with boundary has finite perimeter the boundary conformal
factor $\beta \in \RL^1(\del
\Omega_\CC)$, and the boundary length measure $\rd \ell_g$ on $\del \Omega$ is the
pushforward measure $\phi_*(\beta \rd \ell_{g_\CC})$ \cite[Theorem
6.8]{Pommerenkebook}.
Integrability properties of the conformal factors $\beta$ and $\eta$ are controlled
by the regularity of the boundary $\del \Omega$. This motivates the following
definition of regularity classes.
\begin{defi}
    \label{def:bcr}
    Let $\CC$ be a conformal class, $\Omega \in \CC$, and $\CX(\Omega_\CC)$ and $\CY(\del
    \Omega_\CC)$ be function spaces. We say
  that $\Omega$ has \emph{boundary conformal regularity } $\CY$ if for some conformal
  diffeomorphism $\phi : \Omega_\CC \to \Omega$, the boundary conformal factor
  $\beta \in \CY$. We say that $\Omega$ has \emph{interior conformal regularity}
  $\CX$ if for some conformal diffeomorphism the interior conformal factor $\eta
  \in \CX$. 
\end{defi}
We note that our definition of interior conformal regularity differs from that
in \cite{gu,gpu} by a factor of $2$ since in those papers it was stated in terms
of the integrability of
$\abs{\rd \phi}$ rather than $\abs{\rd \phi}^2$. In other words, interior
conformal regularity $\RL^p$ corresponds to domains which are $2p$-conformally
regular in their definitions.

The integrability class of the interior conformal factor 
$\Omega$ has been used to investigate the  properties of the Neumann Laplacian, see
\cite{gu,gpu}. As expected,  the regularity of 
the boundary conformal factor also appears in the study of the Steklov
problem. 
\begin{remark}
  \label{rem:exampleconformalreg}
  By the Kellogg--Warschawski Theorem \cite[Theorem 3.6]{Pommerenkebook},
  surfaces with boundary of class $\RC^{n,\alpha}$, $n \ge 1$ and $0 < \alpha <
  1$, have boundary conformal regularity $\RC^{n-1,\alpha}$. It follows furthermore from
  the arguments in the proof of \cite[Lemma 5.1]{baratchartbourgeoisleblond}
  that any surface with Lipschitz boundary (or, more generally, a chord--arc
  domain, see Section \ref{sec:examples})  has boundary conformal regularity $\RL^p$ for
  some $p > 1$, and surfaces of finite perimeter have boundary conformal regularity
  $\RL^1$.
\end{remark}

Note that domains with exterior cusps do not have boundary conformal regularity $\RL^p$
for any $p>1$. In order to include some of these domains in our analysis we need to recall the
following definition \cite[Sections IV.6, IV.8]{bennettsharpley}.

\begin{defi}
\label{logl}
Given $(\Xi,\mu)$ a measure space of finite measure and $a \ge 0$
the space $\RL (\log \RL)^a (\Xi)$  is a space of functions $f$
on $\Xi$ such that
$$\int_\Xi |f| (\log (2+|f|)^a \de \mu < \infty$$
endowed with the norm
$$
\|f\|_{\RL (\log \RL)^a (\Xi)}=\inf \left\{t>0 : \int_\Xi \abs{\frac{f}{t}}
\left(\log \left(2+\left|\frac{f}{t}\right|\right)\right)^a \de \mu \le
1\right\}.
$$
One can show that $\RL (\log \RL)^a(\Xi)$ is a Banach space for
every $a \ge 0$. The dual of $\RL (\log \RL)^a(\Xi)$ is the space $\exp
\RL^{1/a}(\Xi)$ of functions $f$ on $\Xi$ such that
\begin{equation}
    \label{eq:expl}
    \norm{f}_{\exp \RL^{1/a}(\Xi)} := \inf\set{t > 0 : \int_\Xi
    \exp\left(\abs{\frac{f}{t}}^{1/a} \de \mu \right) \le 1}
    < \infty.
\end{equation}
The norm \eqref{eq:expl} is equivalent to the dual norm on $\exp \RL^{1/a}$
so that there is $C$ depending only on $(\Xi,\mu)$ such that for all $f \in \RL
(\log \RL)^a(\Xi)$ and $u \in \exp \RL^{1/a}(\Xi)$ the H\"older-type inequality
\begin{equation}
\label{eq:Holder}
    \norm{fu}_{\RL^1(\Xi)} \le C \norm{f}_{\RL (\log \RL)^a(\Xi)} \norm u_{\exp
    \RL^{1/a}(\Xi)}
\end{equation}
holds.
\end{defi}
Given a conformal class $\CC$ and $\Omega \in \CC$ with boundary conformal factor
$\beta$ we consider the weighted Steklov problem
\begin{equation}
  \label{prob:weightedSteklov}
  \begin{cases}
    \Delta u = 0 & \text{in } \Omega_\CC; \\
    \del_\nu u = \beta \sigma u & \text{on } \del \Omega_\CC,
  \end{cases}
\end{equation}
The spectrum of this weighted problem is discrete and accumulates at infinity if the trace
$\RW^{1,2}(\Omega_\CC) \hookrightarrow \RL^2(\del \Omega_\CC, \beta \rd \ell_{g_\CC})$
is compact, see \cite[Sections 3 and 4]{GKL}. This is the case if $\beta
\in \RL \log \RL(\Omega_\CC)$, see Proposition \ref{prop:compacttrace}.

\begin{theorem}
    \label{thm:isospectral}
    Let $\CC$ be a conformal class and $(\Omega,g) \in \CC$ be a surface with
    boundary conformal regularity $\RL \log \RL$. Then, 
problems \eqref{prob:Steklov} and
  \eqref{prob:weightedSteklov} are \emph{isospectral} in the sense
  that $\sigma_k(\Omega) = \sigma_k(\Omega_\CC,\beta)$ for all $k \in \N$.  
\end{theorem}

\begin{remark}
    \label{rem:isospectral}

A simple computation shows that Theorem \ref{thm:isospectral} holds when the
 boundary is smooth, see e.g. \cite[Lemma 3.3]{JolSharaf}.
If the boundary is sufficiently rough, isospectrality  
is not a priori  clear,  and is resolved through the study of composition
    operators between Sobolev spaces with appropriately chosen norm.  
In a similar way, the existence of those bounded composition operators gives
    rise to isospectrality of  the weighted Neumann problems, see  \cite{gu}.
This issue was not previously addressed in the literature on the Steklov problem on Lipschitz domains, 
cf. \cite[Proposition 2.1.4]{gp2}.  

\end{remark}

Our main technical theorem is the following.

\begin{theorem} 
  \label{thm:maingeneral}
  Let $\Omega$ be a surface with boundary of boundary conformal regularity $\RL \log \RL$.
  Then its Steklov eigenvalues satisfy the asymptotic formula  \eqref{eq:weyllaw}. 

  Equivalently, for any surface $\Omega$ with smooth boundary and $\beta \in \RL \log
  \RL(\del \Omega)$, the eigenvalues of the weighted problem
  \eqref{prob:weightedSteklov} satisfy the asymptotic formula \eqref{eq:weyllaw}, with
  $\Vol_{d-1}(\del \Omega)$ replaced by $\int_{\del \Omega} \beta \de \ell_g$. 
\end{theorem}
The equivalence of the two formulations follows from Theorem \ref{thm:isospectral}.
Note that in view of  Remarks \ref{rem:exampleconformalreg}, Theorems
\ref{thm:isospectral} and 
\ref{thm:maingeneral} imply Theorem \ref{thm:main}.

\subsection{Variational characterisation and natural domains for the Steklov
problem}

\label{sec:variational}

On a surface with boundary $\Omega$, consider the Sobolev space
\begin{equation}
    \RW^{1,2}(\Omega) := \set{f  \in \RL^2(\Omega) : \abs{\nabla f} \in
    \RL^2(\Omega)},
\end{equation}
where $\nabla f$ is the weak gradient. If $\Omega$ is a surface with Lipschitz boundary, there are two equivalent norms
on $\RW^{1,2}(\Omega)$:
\begin{equation}
    \norm{f}^2_{\RW^{1,2}(\Omega)} = \int_\Omega \abs{\nabla f}^2 \de v_g +
    \int_\Omega f^2 \de v_g,
    \label{eq:interiornorm}
\end{equation}
and 
\begin{equation}
    \label{eq:bdrynorm}
    \norm{f}^2_{\RW^{1,2}_\del(\Omega)} = \int_\Omega \abs{\nabla f}^2 \de v_g +
    \int_{\del \Omega} f^2 \de \ell_g.
\end{equation}
The norm \eqref{eq:interiornorm} is the standard one and is commonly used in interior problems, for
instance the Neumann problem. On the other hand the norm \eqref{eq:bdrynorm} is
a natural
norm of choice for the Steklov problem. When the boundary is only some
collection of Jordan curves these norms may not be equivalent, even when the
boundary has finite perimeter (one can show that this the case for domains with
fast cusps as defined in subsection \ref{subsec:cusps}). By the Meyers--Serrin
theorem, for any surface with boundary $\Omega$ the space
$\RW^{1,2}(\Omega)$ is the completion of
\begin{equation}
    \label{eq:denseinside}
    \CW(\Omega) := \set{f \in \RC^\infty(\Omega) : \norm{f}_{\RW^{1,2}(\Omega)} < \infty}
\end{equation}
under the $\norm \cdot_{\RW^{1,2}(\Omega)}$ norm, which motivates the following
definition.
\begin{defi}
    \label{def:bdrysobolev}
    Let $\Omega$ be a surface with boundary. The boundary Sobolev space
    $\RW^{1,2}_\del(\Omega)$ is defined as the completion of
    \begin{equation}
        \CW_\del(\Omega) :=    \set{f : \overline \Omega \to \R : f \in
            \RC^\infty(\Omega) \text{ and } \norm{f}_{\RW^{1,2}_\del(\Omega)} <
        \infty}
    \end{equation}
    under the $\norm \cdot_{\RW^{1,2}_\del(\Omega)}$ norm.
\end{defi}
Again, for surfaces with sufficiently regular boundary, $\RW^{1,2}(\Omega)$ and
$\RW^{1,2}_\del(\Omega)$ are isomorphic. We give the following condition for
their equivalence in terms of interior and boundary conformal regularity.
\begin{prop}
    \label{prop:isom}
    Let $\Omega$ be a surface with boundary with both interior \emph{and}
    boundary conformal regularity $\RL \log \RL$. Then, $\RW^{1,2}(\Omega)$ and
    $\RW_\del^{1,2}(\Omega)$ are isomorphic.
\end{prop}
The appropriate space to define the
Steklov problem (especially when they are not isomorphic) is $\RW^{1,2}_\del(\Omega)$, see \cite{nazarovtaskinen}.
The Steklov eigenvalues $\sigma_k(\Omega)$ satisfy the variational characterisation
\begin{equation}
    \sigma_k(\Omega) = \inf_{E_k} \sup_{u \in E_k \setminus \set 0}
    \frac{\int_\Omega \abs{\nabla u}^2 \de v_g}{\int_{\del \Omega} u^2 \de
    \ell_g},
\end{equation}
where $E_k$ is a $k+1$ dimensional subspace of $\RW_\del^{1,2}(\Omega)$. For the
weighted problem on $\Omega_\CC$, we have that for $\beta \in \RL \log
\RL(\Omega_\CC)$ the weighted Steklov eigenvalues satisfy the characterisation
\begin{equation}
    \sigma_k(\Omega_\CC,\beta) = \inf_{E_k} \sup_{u \in E_k \setminus \set 0}
    \frac{\int_{\Omega_\CC} \abs{\nabla u}^2 \de v_g}{\int_{\del \Omega_\CC} u^2 \beta \de
    \ell_g},
\end{equation}
where again $E_k$ is a $k+1$ dimensional subspace of
$\RW_\del^{1,2}(\Omega_\CC)$. The isospectrality Theorem \ref{thm:isospectral}
is a consequence of the following result on composition operator between Sobolev
spaces.
\begin{prop}
    \label{prop:comp}
    Let $\CC$ be a conformal class and $(\Omega,g) \in \CC$ be a surface with
    boundary. Let $\phi : \Omega_\CC \to \Omega$ be a conformal diffeomorphism
    with boundary conformal factor $\beta \in \RL \log \RL$. Then, the composition
    operator
    \begin{equation}
        \phi^* : \RW^{1,2}_\del(\Omega) \to \RW^{1,2}_\del(\Omega_\CC) \qquad
        \phi^* f := f \circ \phi
    \end{equation}
    induced by $\phi$ is an isomorphism.
\end{prop}

\subsection*{Plan of the paper}
The paper is organised as follows.
Section \ref{sec:isospectrality} is concerned with the proof of Theorem
\ref{thm:isospectral}. We then use the variational isospectrality to prove that
domains of boundary conformal regularity $\RL \log \RL$ have discrete Steklov
spectrum. Section \ref{sec:asymptotics} is dedicated to proving
Theorem \ref{thm:maingeneral}, expanding on the theory of spectral asymptotics
for variational eigenvalues developed in \cite{birmansolomjaksobolev,Suslina1999}.
In Section \ref{sec:examples} we give a few examples of domains satisfying the
hypotheses of Theorem \ref{thm:maingeneral}. Finally,  in Section
\ref{sec:remarks} we discuss some further extensions and applications of our methods, in particular to the Steklov problem with an indefinite weight and to the Neumann eigenvalue problem.

\subsection*{Acknowledgements}
The authors would like to thank G. Rozenblum for  helpful comments and proposed simplifications of the 
proofs of Proposition~\ref{prop:compacttrace} and Lemma \ref{lem:controlledcounting}.
We are also grateful to   A.~Ukhlov for useful discussions. The research of M.K. was partially supported by NSF grant DMS-2104254.
The research of J.L. was partially supported by EPSRC award EP/T030577/1. He
also thanks  the University of Bristol for its hospitality while this paper was
written. The research of I.P. was partially supported by NSERC and FRQNT.

\section{Isospectrality and composition operators}

\label{sec:isospectrality}

We first prove the following lemma about composition operators on some
Orlicz spaces, in similar fashion to \cite[Theorem 4]{gu} which
is stated for Lebesgue spaces.

\begin{lemma}
    For $j \in \set{1,2}$, let $(\Xi_j,\mu_j)$ be measure spaces with finite
    measure,  
    $\phi : \Xi_1 \to \Xi_2$ be measurable and suppose that the pushforward measure
    $\phi_*(\mu_1) = \beta \mu_2$, where $\beta : \Xi_2 \to (0,\infty)$. Then,
    $\phi$ induces a bounded composition operator
    \begin{equation}
        \phi^* : \exp \RL^2(\Xi_2) \to \RL^2(\Xi_1), \qquad \phi^* f := f \circ
        \phi
    \end{equation}
    if and only if $\beta \in \RL \log \RL(\Xi_2)$.
\end{lemma}

\begin{proof}
    To prove that the condition  $\beta \in \RL \log \RL(\Xi_2)$ is sufficient, assume that $f \in \exp \RL^2(\Xi_2)$, so that $\abs f^2 \in \exp
    \RL(\Xi_2)$. Since $\RL \log \RL(\Xi_2)$ is a reflexive space with dual
    $\exp \RL(\Xi_2)$, we can compute
    \begin{equation}
        \begin{aligned}
            \label{eq:boundedcomp}
            \int_{\Xi_1} \abs{\phi^* f}^2 \de \mu_1 &=  \int_{\Xi_2} \abs f^2
            \beta \de \mu_2 \\
            &\le 
            \norm{\beta}_{\RL \log \RL(\Xi_2)} 
            \norm{f^2}_{\exp \RL(\Xi_2)}\\
            &=  \norm{\beta}_{\RL \log \RL(\Xi_2)} \norm{f}_{\exp
            \RL^2(\Xi_2)}^2. 
    \end{aligned}
    \end{equation}
    Let us now show that the condition is necessary. Indeed, if $\beta \not \in \RL \log \RL(\Xi_2)$,
    it is not a bounded linear functional on $\exp \RL(\Xi_2)$, so we can find
    $\abs f^2 \in \exp \RL(\Xi_2)$ so that the second term on the first line in
    \eqref{eq:boundedcomp} is not finite.
\end{proof}

In the next proposition, we show compactness of a weighted boundary trace. The
proof is similar in nature to the ideas in \cite[Example 3.19 (iii)]{GKL} where
the weight is instead in the interior.

\begin{prop}
  \label{prop:compacttrace}
  Let $\Omega$ be a surface with smooth boundary and $0 \le \beta \in \RL \log
  \RL(\del \Omega)$, $\beta \not \equiv 0$. Then, the trace $T_\beta : \RW^{1,2}(\Omega)
  \to \RL^2(\del \Omega, \beta \rd \ell)$ is compact.
\end{prop}

\begin{proof}
    Define $\theta : \del \Omega \to \R$ as $\theta := \frac 1 \beta
    \bone_{\set{\beta > 0}}$. 
 Consider the diagram 
\begin{equation}
  \label{eq:compacttrace}
  \begin{tikzcd}
    \RW^{1,2}(\Omega) \arrow{r}{T} \arrow{drr}[swap]{T_\beta} & \exp \RL^2(\del
    \Omega)  \arrow{r}{M_{\sqrt \beta}} &
    \arrow{d}{M_{\sqrt \theta}} \RL^2(\del \Omega)\\
   & & \RL^2(\del \Omega,\beta \rd \ell),
  \end{tikzcd}
\end{equation}
where $T$ is the trace and $M_h$ is the operator of multiplication by the
function $h$. The trace operator $T$ is bounded; $\exp \RL^2(\del \Omega)$ is in fact the optimal
target space on $\del \Omega$ for bounded traces from $\RW^{1,2}(\Omega)$, see
\cite[Example 5.3]{cianchipick}. By H\"older inequality \eqref{eq:Holder}
between $\exp L(\del \Omega)$ and $\RL
\log \RL(\del \Omega)$ there exists $C > 0$ such that
\begin{equation}
    \norm{M_{\sqrt \beta} f}_{\RL^2(\del \Omega}^2 = \int_{\del \Omega} f^2
        \beta \de \ell \le
  C\norm{f^2}_{\exp \RL(\del \Omega)} \norm{\beta}_{\RL \log \RL(\del \Omega)} =
  C\norm{f}_{\exp \RL^2(\del \Omega)}^2 \norm{\beta}_{\RL \log \RL(\del \Omega)}.
\end{equation}
In other words, $M_{\sqrt{\beta}}$ is bounded with norm at most $C \norm{\beta}_{\RL
\log \RL(\del \Omega)}$. As for $M_{\sqrt{\theta}}$, we have that
\begin{equation}
    \norm{M_{\sqrt \theta}f}^2_{\RL^2(\del \Omega;\beta \de \ell)} = \int_{\del
    \Omega \cap \set{\beta > 0}} f^2 \de \ell \le \norm{f}_{\RL^2(\del
    \Omega)}^2
\end{equation}
Thus, for a probably different constant $C>0$
independent of $\beta$ we have
\begin{equation}
    \norm{T_\beta} \le C \norm{\beta}_{\RL \log \RL(\del \Omega_\CC)}.
\end{equation}
To prove compactness, it is sufficient to prove that $M_{\sqrt{\beta}} \circ T$
is compact; if $\beta$ is a nonnegative smooth function this follows from the
usual trace restriction theorem. 
By density of smooth functions in $\RL \log \RL$, for every $\eps > 0$, there is
a nonnegative $\beta_\eps \in \RC^\infty(\del \Omega)$ such that $\norm{\beta -
\beta_\eps}_{\RL \log \RL(\del \Omega)} < \eps$ and $\beta_\eps \le \beta$
almost everywhere, so that $\sqrt{\beta} - \sqrt{\beta_\eps} \le \sqrt{\beta -
\beta_\eps}$. But then, 
\begin{equation}
    \norm{M_{\sqrt \beta} \circ T - M_{\sqrt{\beta_\eps}} \circ T}  \le
    \norm{M_{\sqrt{\beta - \beta_\eps}}} \norm T \le \norm{\beta - \beta_\eps}_{\RL \log
    \RL(\del \Omega)} \norm T \le \eps \norm T.
\end{equation}
Thus $M_{\sqrt{\beta}} \circ T$ is a norm limit of compact operators
hence compact itself and $T_\beta$ is compact also.
\end{proof}

We now have the right tools to prove the composition Propositon \ref{prop:comp},
following the structure of the proof of \cite[Theorem 6]{gu}.

\begin{proof}[Proof of Proposition \ref{prop:comp}]
    Let $f \in \CW_\del(\Omega)$ as in Definition \ref{def:bdrysobolev}. Invariance of the
    Dirichlet energy under conformal diffeomorphisms tells us that
    \begin{equation}
        \norm{\nabla f}_{\RL^2( \Omega)} = \norm{\nabla(\phi^*
        f)}_{\RL^2(\Omega_\CC)}.
    \end{equation}
    Since the boundary conformal factor is in $\RL \log \RL$, $\phi$ induces the
    bounded composition operator $(\phi^{-1})^* : \exp \RL^2(\del \Omega_\CC) \to
    \RL^2(\del \Omega)$. Therefore, for every $a \in \R$, we have that
    \begin{equation}
\label{eq:consta}
        \begin{aligned}
        \abs a = \Per(\Omega)^{-1/2} \norm{c}_{\RL^2(\del \Omega)} &\le
        \Per(\Omega)^{-1/2} \left(\norm{f}_{\RL^2(\del \Omega)} + \norm{f-
        a}_{\RL^2(\del \Omega)} \right) \\ &\le 
        \Per(\Omega)^{-1/2} \left(\norm{f}_{\RL^2(\del \Omega)} + \norm{\phi^* f-
        a}_{\exp \RL^2(\del \Omega_\CC)} \right).
    \end{aligned}
    \end{equation}
   We claim that
    there exists  a constant $C > 0$ such that
    \begin{equation}
        \begin{aligned}
        \norm{\phi^* f}_{\RL^2(\del \Omega_\CC)} &\le \inf_{a \in \R}
        \left(\norm{a}_{\RL^2(\del \Omega_\CC)} + \norm{\phi^*f - a}_{\RL^2(\del
        \Omega_\CC)}\right)
        \\
        &\le \inf_{a \in \R} \left(\Per(\Omega)^{-1/2} \Per(\Omega_\CC)^{1/2} +
        C\norm{1}_{\RL \log \RL(\del \Omega_\CC)}\right) \left( \norm{f}_{\RL^2(\del \Omega)} +
        \norm{\phi^*f - a}_{\exp \RL^2(\del \Omega_\CC)} \right).
    \end{aligned}
    \end{equation}
Indeed, the first inequality is just the triangle inequality.  We then use \eqref{eq:consta} to estimate the first term,  and  inequality \eqref{eq:Holder}
together with the relations $ \norm{h}_{\RL^2}^2=\norm{h^2}_{\RL^1}$, $ \norm{h}_{\exp \RL^2}^2=\norm{h^2}_{\exp \RL^1}$  to estimate the second.

    The space $\exp \RL^2(\del \Omega_\CC)$ is the optimal
    target space for traces from $\RW^{1,2}(\Omega_\CC)$, and this is equivalent
    to the validity of a Poincar\'e trace inequality, see \cite[Theorems 1.3 and
    5.3]{cianchipick},
    \begin{equation}
        \inf_a \norm{\phi^* f - a}_{\exp \RL^2(\del \Omega_\CC)} \le C \norm{\nabla
        \phi^* f}_{\RL^2(\Omega_\CC)}.
    \end{equation}
    Combining the previous display formulas  gives us the existence of some constant $C > 0$ such
    that
    \begin{equation}
    \norm{\phi^*f}_{\RW_\del^{1,2}(\Omega_\CC)} \le C \norm{f}_{\RW_\del^{1,2}(\Omega)}.
\end{equation}
Since $\CW_\del(\Omega)$ is dense in $\RW^{1,2}_\del(\Omega)$, the pullback
$\phi^*$ extends to the whole space as a bounded operator as well.
Proving the analogous result for $(\phi^{-1})^*$  is simpler. Since $\Omega_\CC$ has smooth boundary, the spaces  $\RW^{1,2}(\Omega_\CC)$
and $\RW^{1,2}_\del(\Omega_\CC)$ are isomorphic. Compactness (in fact, boundedness is enough here)  of the trace
$\RW^{1,2}(\Omega_\CC) \to \RL^2(\del \Omega_\CC, \beta \rd \ell_{g_\CC})$
obtained in Proposition \ref{prop:compacttrace} then
implies that for every $h \in \CW(\Omega_\CC)$,
\begin{equation}
    \norm{(\phi^{-1})^* h}_{\RW_\del^{1,2}(\Omega)}^2 =\norm{\nabla
        h}_{\RL^2(\Omega_\CC)}^2 + \norm{h}_{\RL^2(\del \Omega_\CC, \beta \rd \ell_{g_\CC})}^2
\le C    \norm{h}_{\RW_\del^{1,2}(\Omega_\CC)}^2.
\end{equation}
By density we once again have that $(\phi^{-1})^*$ extends to the whole space as
a bounded operator, completing the proof.
\end{proof}

We can now prove Theorem \ref{thm:isospectral}.

\begin{proof}[Proof of Theorem \ref{thm:isospectral}]
    Let $E_k$ be a $k+1$ dimensional subspace of $\RW_\del^{1,2}(\Omega)$. Then,
    by Proposition \ref{prop:comp},
    $\phi^*(E_k)$ is a $k+1$ dimensional subspace of
    $\RW_\del^{1,2}(\Omega_\CC)$, and for every $u \in E_k$,
    \begin{equation}
        \frac{\int_\Omega \abs{\nabla u}^2 \de v_g}{\int_{\del \Omega} u^2 \de
        \ell_g}=
        \frac{\int_{\Omega_\CC} \abs{\nabla u}^2 \de v_{g_\CC}}{\int_{\del
            \Omega_\CC} u^2 \beta \de
        \ell_{g_\CC}}.
    \end{equation}
    This implies directly that $\sigma_k(\Omega_\CC,\beta) \le
    \sigma_k(\Omega)$.  The analogous reasoning with $(\phi^{-1})^*$ instead  of $\phi^*$ gives
    the reverse inequality.
\end{proof}

In order to prove Proposition \ref{prop:isom},  we extend the
results of \cite{gu} to a slightly more singular interior conformal factor.

\begin{lemma}
    Let $\CC$ be a conformal class and $(\Omega,g) \in \CC$ be a surface with
    boundary, which has
    interior conformal regularity $\RL \log \RL$ through the conformal
    diffeomorphism $\phi : \Omega_\CC \to \Omega$. Then, the composition
    operator
    \begin{equation}
        \phi^* : \RW^{1,2}(\Omega) \to \RW^{1,2}(\Omega_\CC) \qquad
        \phi^* f := f \circ \phi
    \end{equation}
    induced by $\phi$ is an isomorphism.
\end{lemma}

\begin{proof}
    The proof is essentially identical to the proof of Proposition \ref{prop:comp}. We replace the
    result on the traces $\RW_\del^{1,2}(\Omega_\CC) \to
    \exp \RL^2(\del \Omega_\CC)$  and the corresponding Poincar\'e inequality with the optimal Sobolev embedding $\RW^{1,2}(\Omega_\CC) \to \exp
    \RL^2(\Omega_\CC)$, and use the fact that the interior conformal factor
    $\abs{\rd\phi}^2 \in \RL \log \RL$ to get a bounded composition operator $\exp
    \RL^2(\Omega_\CC) \to \RL^2(\Omega)$.
\end{proof}

We can now prove Proposition \ref{prop:isom}.

\begin{proof}[Proof of Proposition \ref{prop:isom}]
    Since $\Omega_\CC$ is a surface with smooth boundary, the spaces
    $\RW^{1,2}(\Omega_\CC)$ and $\RW_\del^{1,2}(\Omega_\CC)$ are isomorphic, via
    some linear map $\iota$. Interior
    conformal regularity $\RL \log \RL$ provides us with an isomorphism
    $\phi^* : \RW^{1,2}(\Omega) \to \RW^{1,2}(\Omega_\CC)$ and boundary
    conformal regularity $\RL \log \RL$ with an isomorphism $\phi_\del^* :
    \RW_\del(\Omega) \to \RW^{1,2}(\Omega_\CC)$. The composition
    $(\phi_\del^*)^{-1} \circ \iota \circ \phi^*$ provides the desired
    isomorphism.
\end{proof}

\section{Spectral asymptotics}
\label{sec:asymptotics}
\subsection{Eigenvalue counting functions of compact operators}
We first  present  some  known results about spectral asymptotics of compact
operators defined via quadratic forms. These  results can be found in the works
of Suslina
\cite{Suslina1999}, Birman--Solomyak \cite{birmansolomjaksobolev}, and
Sukochev--Zanin \cite{sukochevzanin}, in a  more general form. For the convenience of the reader we
state them here in a form which is specific for  our purposes.

Let $\CH$ be a Hilbert space and $K \in \CK(\CH)$ be a self-adjoint nonnegative compact
operator. The non-zero spectrum of $K$ consists of a discrete set of
nonincreasing nonnegative
eigenvalues $\set{\lambda_j(K) : j \in \N}$ counted with multiplicity and
 converging to $0$. 
The variational characterisation of the eigenvalues yields
\begin{equation}
  \lambda_j(K) = \max_{E_j \subset \CH} \min_{u \in E_j \setminus \set 0} \frac{(Ku,u)}{(u,u)},
  \label{eq:Rayleigh}
\end{equation}
where $E_j$ ranges over $j$ dimensional subspaces. Note that the operator $K$ can be equivalently defined via the associated bilinear form appearing in the numerator \eqref{eq:Rayleigh}; we will use this observation further on.
For $\lambda >0$, define the eigenvalue counting function
\begin{equation}
  n(\lambda;K) := \# \set{j : \lambda_j(K) > \lambda},
\end{equation}
and, for a given $\alpha > 0$, the functionals
\begin{equation}
  \overline n_\alpha(K) = \limsup_{\lambda \searrow 0} \lambda^\alpha
  n(\lambda;K) \text{ and } \underline n_\alpha(K) = \liminf_{\lambda \searrow
  0} \lambda^\alpha n(\lambda;K).
\end{equation}
Note that if $\overline n_\alpha(K) = \underline n_\alpha(K) = C_\alpha$, then
\begin{equation}
    n(\lambda;K) = C_\alpha \lambda^{-\alpha}(1 + \smallo 1).
\end{equation}
We make use of the following general properties of  this counting function,
which are collected in \cite[Appendix 1]{birmansolomjaksobolev}, see also the
references therein.

\begin{lemma}
  \label{lem:countingproperties}
  The following properties hold:
  \begin{enumerate}
    \item \label{pr:compactpert} \cite[Lemma 1.16]{birmansolomjaksobolev} For any $\alpha > 0$, the functionals $\overline n_\alpha(K)$ and
      $\underline n_\alpha(K)$ 
      are invariant under compact perturbations of the inner product on $\CH$,
      as well as restriction to subspaces of finite codimension.
    \item \label{pr:closeness} \cite[Lemma 1.18 and its
        proof]{birmansolomjaksobolev}, \cite[Lemma 1.5]{Suslina1999}, Weyl--Fan
        Ky lemma. Let $K_1 \le K_2 \in \CK(\CH)$ be nonnegative self-adjoint compact
      operators. Then, for any $\alpha>0$, 
      \begin{equation}
        \abs{\overline n_\alpha(K_1)^{\frac{1}{1 + \alpha}} -
        \overline n_\alpha(K_2)^{\frac{1}{1 + \alpha}} } \le \overline
        n_\alpha(K_2 - K_1)^{\frac{1}{1+\alpha}}
      \end{equation}
      and
      \begin{equation}
        \abs{\underline n_\alpha(K_1)^{\frac{1}{1 + \alpha}} -
        \underline n_\alpha(K_2)^{\frac{1}{1 + \alpha}} } \le \overline
        n_\alpha(K_2 - K_1)^{\frac{1}{1+\alpha}}.
      \end{equation}
    \item \label{pr:diffspaces} \cite[Lemma 1.15]{birmansolomjaksobolev} Let $K_1 \in \CK(\CH_1)$ and $K_2 \in
      \CK(\CH_2)$ be nonnegative self-adjoint compact operators. Let $B :
      \CH_1 \to \CH_2$ be a bounded operator such that $(K_1 u,u)_{\CH_1} = 0$
      for all $u \in \ker B$. If there is $a > 0$ such that for all $u \in
      \CH_1\setminus \ker B$
      \begin{equation}
        \frac{(K_1 u,u)_{\CH_1}}{(u,u)_{\CH_1}} \le a 
        \frac{(K_2 B u,B u)_{\CH_2}}{(Bu,Bu)_{\CH_2}},
      \end{equation}
      then for all $\lambda > 0$, $n(\lambda;K_1) \le n(a^{-1}\lambda;K_2)$ for
      all $\lambda > 0$.
  \end{enumerate}
\end{lemma}

We will use these abstract results in the concrete situation where $\CH = \RW^{1/2,2}(\Gamma)$, where
$\Gamma$ is a finite collection of smooth curves with length measure $\rd \ell$. For $\beta : \Gamma \to
[0,\infty)$ let $K_\beta$ be the operator in $\RW^{1/2,2}(\Gamma)$ be defined by
the bilinear form
\begin{equation}
  \label{eq:bilinear}
  (K_{\beta} u, v)_{\CH} = \int_\Gamma u \overline v \beta \de \ell, \quad u,v
  \in \operatorname{Dom}(K_{\beta}).
\end{equation}
The following lemma is a direct reinterpretation of \cite[Lemma
4.4]{sukochevzanin} (see also \cite[Theorem 2.1]{RozenShar}) in view of the variational characterisation of the
eigenvalue counting function \cite[Lemma 1.14]{birmansolomjaksobolev}.
\begin{lemma}
    Let $\Gamma$ be a finite collection of smooth curves and $0 \le \beta \in \RL
  \log \RL(\Gamma)$. Let $K_{\beta}$ be the
  self-adjoint operator on $\RW^{1/2,2}(\Gamma)$ defined via the
  bilinear form \eqref{eq:bilinear}. Then  there exists a constant $C(\Gamma)>0$ such that
  \begin{equation}
    n(\lambda;K_{\beta}) \le C(\Gamma) \lambda^{-1} \norm{\beta}_{\RL \log
    \RL(\Gamma)}.
  \end{equation}
  \label{lem:controlledcounting}
\end{lemma}

We now have the required tools to prove Theorem \ref{thm:maingeneral}.

\subsection{Proof of Theorem \ref{thm:maingeneral}}
  We turn to the second, equivalent, statement. Recall that the eigenvalues of the weighted
  Steklov problem on a surface with smooth boundary are characterised
  variationally as
  \begin{equation}
    \label{eq:rayleighdtn}
\sigma_k(\Omega,\beta) = \min_{E_k} \max_{U \in E_k \setminus \set 0}
    \frac{\int_{\Omega} \abs{\nabla U}^2 \de A_{g}}{\int_{\del
    \Omega} u^2 \beta \de \ell_{g}}.
  \end{equation}
  Here $E_k \subset \RW^{1,2}(\Omega)$ (which is isomorphic to
  $\RW^{1,2}_\del(\Omega)$) is a $k+1$
     dimensional subspace, and $u:=\tau U$, where   $\tau : \RW^{1,2}(\Omega)
     \to \RW^{1/2,2}(\del \Omega)$ is the trace operator, which is continuous.
     Here and further on we adopt the following convention: capital letters
     denote functions in the interior, and the corresponding lower case letters
     denote their boundary traces.

Let
 \begin{equation}
     \label{eq:cx}
    \CX := \set{V \in \RW^{1,2}(\Omega) : \int_{\del \Omega} v \beta \de
    \ell_{g} = 0}.
  \end{equation}
  be the orthogonal complement in $\RL^2(\del \Omega;\beta \de
  \ell_{g})$ to the kernel of the weighted Dirichlet-to-Neumann map, i.e. to the constant functions.
 We equip $\CX$ with the inner product
  \begin{equation}
    (U,U)_\CX = \int_{\Omega} \abs{\nabla U}^2 \de A_{g}.
  \end{equation}
Let us define an operator $Q_\beta$ on $\CX$ via the bilinear form 
\begin{equation}
\label{eq:Qbeta}
(Q_\beta U, V)_{\CX} = \int_{\partial \Omega} u \, v \, \beta \de \ell_g, \quad u,v \in \CX.
\end{equation}
Clearly,  we have that
  \begin{equation}
    \label{eq:maxmin}
    \lambda_k(Q_\beta) = \max_{E_k \subset \CX} \min_{u \in E_k \setminus \set 0}
    \frac{\int_{\del \Omega} u^2 \beta \de \ell_{g}}{(U,U)_\CX}.
  \end{equation}
  In view of \eqref{eq:maxmin} and \eqref{eq:rayleighdtn} we have that for $k \ge 1$,
  $\sigma_k(\Omega,\beta)^{-1} = \lambda_k(Q_\beta)$, so that 
  \begin{equation}
    \label{eq:equivalence}
    N(\sigma;M,\beta) - 1 = n(\sigma^{-1};Q_\beta),
  \end{equation}
where we have subtracted one  on the left to account for the eigenvalue zero. 
 Let us  find the asymptotics of $n(\sigma^{-1};Q_\beta)$
as $\sigma^{-1} =:\lambda \searrow 0$.  
  It follows from Lemma \ref{lem:countingproperties}\eqref{pr:compactpert} that
  the asymptotics of  $n(\lambda;Q_\beta)$  does not change if we  first
  replace $(U,U)_\CX$
  with $(U,U)_\CX + (U,U)_{\RL^2(\Omega)}$  (this is a compact perturbation), and then lift the orthogonality
  condition,  in order to consider $U \in \RW^{1,2}(\Omega)$ as in
  \eqref{eq:rayleighdtn}. By the density of smooth functions in $\RL \log \RL$,
  for every $\eps > 0$ we can find a smooth $\beta_\eps \in \RC^\infty(\del
  \Omega)$
  such that $\norm{\beta - \beta_\eps}_{\RL \log \RL} < \eps$. Without loss of
  generality, we suppose that $\beta_\eps \le \beta$ almost everywhere so that
  $Q_\beta -
  Q_{\beta_\eps}$ is a positive operator. Since we know, by the general theory
  of pseudodifferential operators, that as $\lambda \searrow 0$
  \begin{equation}
    n(\lambda;Q_{\beta_\eps}) = \frac{\lambda^{-1}}{\pi} \int_{\del \Omega}
    \beta_\eps \de \ell_g
    + \smallo{\lambda^{-1}},
  \end{equation}
  it is sufficient by Lemma \ref{lem:countingproperties}\eqref{pr:closeness} to
  show that
  \begin{equation}
    n(\lambda;Q_\beta - Q_{\beta_\eps}) \le C \lambda^{-1} \norm{\beta -
    \beta_\eps}_{\RL \log \RL},
  \end{equation}
  with $C$ depending only on $\Omega$.  It immediately follows from \eqref{eq:Qbeta} that  $\ker \tau \subset \ker(Q_\beta - Q_{\beta_\eps})$.
  Defining $K_{\beta}$ as in Lemma \ref{lem:controlledcounting} with $\Gamma = \del
  \Omega$, we have that for all $U \in \RW^{1,2}(\Omega)$, 
  \begin{equation}
   ( (K_{\beta} - K_{\beta_\eps}) u,u)_{\RW^{1/2,2}(\del \Omega)} = ((Q_\beta - Q_{\beta_\eps}) U, U)_{\RW^{1,2}(\Omega)}
    \end{equation}
    By the trace theorem, we also have that there exists $C_\Omega$ such that
    \begin{equation}
      (\tau U,\tau U)_{\RW^{1/2,2}(\del \Omega)} \le C_\Omega (U,U)_{\RW^{1,2}(\Omega)}.
    \end{equation}
    By applying first Lemma \ref{lem:countingproperties}\eqref{pr:diffspaces}
    then Lemma \ref{lem:controlledcounting} we  deduce that 
    \begin{equation}
      \begin{aligned}
        n(\lambda;Q_\beta - Q_{\beta_\eps}) &\le n(C_\Omega \lambda;K_{\beta} -
        K_{\beta_\eps}) \\
      &\le C_\Omega' \lambda^{-1} \norm{\beta-\beta_\eps}_{\RL \log \RL(\del
      \Omega)} \\
      &\le C_\Omega' \lambda^{-1} \eps.
      \end{aligned}
    \end{equation}
    Since this holds for arbitrary $\eps > 0$, we deduce that 
    \begin{equation}
      n(\lambda;Q_\beta) = \frac{\lambda^{-1}}{\pi} \int_{\del \Omega} \beta \de
      \ell_g +
      \smallo{\lambda^{-1}}
    \end{equation}
    and in view of  \eqref{eq:equivalence} this completes the proof of the theorem.

\section{Examples}
\label{sec:examples}
In this last section, we  present  examples of domains having conformal
regularity $\RL \log \RL$ and explore the sharpness of  Theorem
\ref{thm:maingeneral}. We give planar domains as example, but they extend in a
straightforward manner to domains in a complete Riemannian surface.

\subsection{Chord-arc domains}

Recall that a Jordan domain $\Omega \subset \mathbb{R}^2$ is called a
\emph{chord-arc} (or \emph{Lavrentiev}) \emph{domain} if there exists a constant
$C$ such that for any $x,y \in \del \Omega$ 
$$\dist_{\del \Omega} (x,y) \le C \dist_{\R^2}(x,y),$$
where the left-hand sides denotes the length of the shortest arc of the boundary
joining $x$ and $y$, and the right-hand side denotes the distance between $x$
and $y$ in $\mathbb{R}^2$.  It is clear that any Lipschitz domain is a chord-arc
domain. The class of chord-arc domains is larger than Lipschitz and includes, in
particular, the domain bounded between two logarithmic spirals.  Note that domains with cusps are not
chord-arc.

There is a large literature on the conformal regularity of chord-arc domains,
see, for instance \cite{Jerison} and references therein.
The following result is well-known. We outline its proof below for the convenience of the reader.
\begin{prop} Let $\Omega \subset \mathbb{R}^2$ be a chord-arc domain and let $\phi :\D \to \Omega$ be a conformal map.
Then $\phi' \in \RL^p(\partial \Omega)$ for some $p>1$. 
\end{prop}
\begin{proof}
Since $\Omega$ is a chord-arc domain, by \cite[Theorem 1]{Zins82} we have that
$\phi' \in \RA_q$ for some $q>1$, where $\RA_q$ denotes a Muckenhoupt class of
weights (see, for instance \cite[Section VI.6]{garnett} for a definition). By
\cite[Corollary 6.10]{garnett}, every Muckenhoupt weight on $\del \Omega$ of
class $\RA_q$, $q > 1$
is in $\RL^p(\del \Omega)$ for some $p > 1$, which is our claim.
\end{proof}
\subsection{Domains with  cusps}
\label{subsec:cusps}
 Let $\Omega \subset \mathbb{R}^2$  be a domain with boundary  $\partial \Omega$ which is a finite 
union of  smooth curves. If two curves meet at an interior angle zero we say
that they form an \emph{outward cusp}, and if the interior angle is equal to
$2\pi$ we say that they form an \emph{inward cusp}.

At the tip  $x_0$ of an inward cusp $\phi'(x_0)=0$, in fact
domains with inward cusps have $\phi' \in \RC^{0,1}(\partial \Omega)$ 
\cite[Theorem 3.9]{Pommerenkebook}. A typical example is the standard cardioid domain defined in polar coordinates
as $\set{(r,\theta) : r=2(1+\cos \theta)}$, for which $\phi(z)=(z+1)^2$. 

Consider now domains with outward cusps.  Suppose that in a neighbourhood of the
outward cusp at $x_0$ the boundary of $\Omega$ consists of two smooth curves
$\gamma_1(t)$, $\gamma_2(t)$, where $t$ is the arc length parameter and
$\gamma_1(0) = \gamma_2(0) = x_0$. We say that $\Omega$
has a \emph{slow} cusp at $x_0$ if there is $\alpha \in (0,1)$ (the speed of the
cusp) such that
\begin{equation}
  \lim_{t \searrow 0}\frac{\abs{\gamma_1(t) - \gamma_2(t)}}{t^{1 + \alpha}} = s_\alpha > 0
\end{equation}
In turn if there is $C$ such that $t^{-2} \abs{\gamma_1(t) - \gamma_2(t)} \le C
< \infty$ for all $t > 0$ we say that $\Omega$ has a \emph{fast} cusp at $x_0$. 

 It is shown in \cite{nazarovtaskinen} that whenever a domain has a fast cusp, 
the Dirichlet-to-Neumann map does not have a compact resolvent. Therefore,  its
spectrum is not discrete and  Weyl's law can not hold.   However, the
Dirichlet-to-Neumann map for a domain whose boundary is Lipschitz except at a
finite number of slow cusps has a compact resolvent,
and hence a discrete spectrum.

Let $\phi : \Omega_{\CC}  \to \Omega$ be a conformal diffeomorphism and let $z_0=
\phi^{-1}(x_0)$ be the pre-image of a cusp of speed $\alpha$. Applying \cite[Proposition 2.10]{kaiserlehner} to \cite[Corollary
1]{prokhorovcusp}, we see that as  $z \to z_0$ the
conformal factor $|\rd \phi(z)|$ behaves asymptotically as
\begin{equation}
  \label{eq:cusp}
  \abs{ \rd \phi(z)}  = O\left( |z - z_0|^{-1} \left(-\log(\abs{z -
    z_0}\right)^{-1-\frac 1 \alpha}\right). 
\end{equation}
A direct calculation gives that $\abs{\rd \phi} \in \RL \log \RL$ if and only if
$0 < \alpha < 1$, in other words precisely when the spectrum is discrete. This shows that Theorem \ref{thm:maingeneral} gives in a sense an optimal condition for the validity of Weyl's law. 

Let us summarize the results of this subsection in the following
\begin{prop}
\label{prop:cusps}
Let $\Omega \subset \mathbb{R}^2$ be a domain  with piecewise smooth  boundary, possibly with interior and exterior cusps. 
If all exterior cusps are slow then Weyl's law \eqref{eq:weyllaw} holds for the counting function of the Steklov eigenvalues on $\Omega$. 
Moreover, if $\Omega$ has at least one fast cusp, then the Steklov spectrum of $\Omega$ is not discrete. 
\end{prop}

\begin{remark}
It would be interesting to understand if there exist domains for which the Steklov spectrum is discrete but the Weyl's law \eqref{eq:weyllaw} does not hold.
 To construct such an example one needs to find a domain $\Omega$ with  the boundary conformal
  factor in $\RL^1 \setminus \RL \log \RL$, and yet for which the 
 resolvent of the Dirichlet-to-Neumann map is still compact.  
We note that in terms of the weighted problem it seems like this would require going beyond the Orlicz scale: indeed,  for
every $0 \le a < 1$ one can find $\beta \in \RL  (\log \RL)^a(\del \Omega)$ 
so that the embedding $\RW^{1,2}(\partial \Omega) \to \RL^2(\partial \Omega,
\beta \rd \ell)$ is not compact, following the proof found in
\cite[Example 3.19]{GKL}.
\end{remark}

\section{Further remarks and extensions}
\label{sec:remarks}
\subsection{The Steklov problem with indefinite weight}
\label{sec:weighted}

Suppose for now that $\Omega$ is a surface with smooth boundary, and given $\beta :
\del \Omega \to \R$ consider the Steklov problem with an indefinite weight:
\begin{equation}
    \label{prob:weightedlip}
    \begin{cases}
        \Delta u = 0 & \text{in } \Omega; \\
        \del_\nu u = \beta \sigma u & \text{on } \del \Omega.
    \end{cases}
\end{equation}
Indefinite eigenvalue problems of this type have been  considered in the literature, see e.g. \cite{Sandgren, Suslina1999B, Suslina1999, Agranovich}.
If $0 \not \equiv \beta \in \RL \log \RL(\del \Omega)$ changes sign on sets of
positive measure in $\del \Omega$, then the non-zero eigenvalues form two sequences
$\set{\sigma_k^\pm(\Omega,\beta) : k \in \N}$ consisting of the positive and
negative eigenvalues, accumulating respectively at $\pm \infty$. To define the
variational principle, let us first denote 
\begin{equation}
    \norm{f}_{\RW_\del^{1,2}(\Omega;\beta)}^2 = \int_\Omega \abs{\nabla f}^2 \de
    v_g + \int_{\del \Omega} f^2 \abs \beta \de \ell_g
\end{equation}
and
\begin{equation}
    \CW_\del(\Omega;\beta) := \set{f : \overline \Omega \to \R : f \in
        \RC^\infty(\Omega) \text{ and } \norm{f}_{\RW_\del^{1,2}(\Omega;\beta)}
    < \infty}.
\end{equation}
We denote by $\RW^{1,2}_\del(\Omega;\beta)$ the closure of
$\CW_\del(\Omega;\beta)$ under the $\norm \cdot_{\RW^{1,2}_\del(\Omega;\beta)}$
norm, and $\CX$ to be the subset of $\RW^{1,2}_\del(\Omega;\beta)$ orthogonal
to $\beta$. 
Following
\cite{birmansolomjaksobolev,Suslina1999}, the non-zero eigenvalues of problem
\eqref{prob:weightedlip} satisfy the
variational principle
\begin{equation}
    \frac{\pm 1}{\sigma^\pm_k(\Omega,\beta)} = \min_{F_k} \max_{u \in F_k
    \setminus \set 0} \pm \frac{\int_{\del \Omega} u^2 \beta \de
    \ell_g}{\int_{\Omega} \abs{\nabla u}^2 \de A_g},
\end{equation}
where $F_k$ is a codimension $k-1$ subspace of $\CX$. Denoting by
\begin{equation}
    N^\pm(\sigma;\Omega,\beta) := \#\set{k : 0 < \pm\sigma_k^\pm(\Omega,\beta)<
    \sigma}
\end{equation}
the counting functions for each of those sequences, it follows from the work of
Birman--Solomyak \cite{birmansolomjakasymp,birmansolomjakproofs} (see
\cite[Theorem 6.1]{ponge} for a modern proof, in English) that if $ \beta$
is smooth, then
\begin{equation}
    \label{eq:weylindefinite}
    N^\pm(\sigma;\Omega,\beta) = \frac{\sigma}{\pi} \int_{\del \Omega} \beta_\pm
    \de \ell + \smallo{\sigma}
\end{equation}
where $\beta_\pm = \max\set{0,\pm \beta}$ are the positive and negative parts of
$\beta$. This formula is valid whether or not $\beta$ takes both positive and
negative values. Using the same methods as in Section \ref{sec:asymptotics} allows us to
extend this result to $\beta \in \RL \log \RL(\del \Omega)$. We note that the
results in Lemma \ref{lem:countingproperties} are in fact proven in
\cite{birmansolomjaksobolev,Suslina1999} for operators
with both positive and negative spectrum, with the obvious redefinition of
the functions $\overline n_\alpha^\pm$ and $\underline n_\alpha^\pm$. 

When
$\Omega$ has non-smooth boundary, we consider once again a conformal map
$\phi : \Omega_\CC \to \Omega$. If the product $\phi^* \beta \abs{\rd \phi} \in
\RL \log \RL(\del \Omega_\CC)$, then the proof of Proposition \ref{prop:isom}
    carries through and  $\phi$ induces an isomorphism $\phi^*$ between
    $\RW^{1,2}_\del(\Omega;\beta)$ and $\RW^{1,2}_\del(\Omega_\CC;\phi^*
    \beta)$. If both $\phi^* \beta$ and $\phi^*\beta \abs{\rd \phi}$ are in $\RL
    \log \RL(\del \Omega_\CC)$ then
    \begin{equation}
        \RW_\del^{1,2}(\Omega_\CC;\phi^*\beta) \cong
        \RW_\del^{1,2}(\Omega_\CC) \cong
        \RW_\del^{1,2}(\Omega_\CC;\phi^*\beta\abs{\rd \phi})
    \end{equation}
    and as in Theorem \ref{thm:isospectral} problem
\eqref{prob:weightedlip} is isospectral to
\begin{equation}
    \begin{cases}
        \Delta u = 0 & \text{in } \Omega_\CC \\
        \del_\nu u =  \abs{\rd \phi} \phi^* \beta \sigma u & \text{on } \del
        \Omega_\CC.
    \end{cases}
\end{equation}
We see directly that a sufficient  condition for having the Weyl law
\eqref{eq:weylindefinite} is also that $\abs{\rd \phi} \phi^* \beta \in \RL \log
\RL(\del \Omega_\CC)$. For any surface with Lipschitz boundary,
we have that $\abs{\rd \phi}$ and $\abs{\rd \phi}^{-1}$ are in $\RL^p$ for some $p > 1$ with H\"older
conjugate $p'$, see \cite[proof of Lemma 5.1]{baratchartbourgeoisleblond}.
Therefore, if $\phi^* \beta \in \RL^q(\del \Omega_\CC)$ for $q >
p'$,  the Weyl law \eqref{eq:weylindefinite} holds. Arguing as in  \cite[Theorem
4]{gu}, one can show that for  $q > p'$  the map  $\phi$ induces a bounded composition operator $\phi^* :
\RL^{\frac{qp}{p-1}}(\partial \Omega) \to \RL^q(\partial \Omega_\CC)$. Therefore, a
sufficient condition for the Weyl law to hold is $\beta \in \RL^r(\del \Omega)$, for some $r >
p^2/(p-1)^2$. In particular, if $\Omega$ is a Lipschitz domain and $\beta$ is in
some Orlicz space contained in $\RL^q(\del \Omega)$ for any  $q < \infty$, then the weighted Steklov problem, definite
or not, satisfies the Weyl law \eqref{eq:weylindefinite}.

\subsection{The Neumann problem}
\label{sec:Neumann}
The methods developed in this paper can be also  applied to the Neumann problem
\begin{equation}
  \begin{cases}
    - \Delta_g u = \lambda u & \text{in } \Omega \\
    \del_\nu u = 0 & \text{on } \del \Omega.
  \end{cases}
\end{equation}
In this case, the conformal map $\phi : \Omega_\CC \to \Omega$ gives rise to the
variationally isospectral weighted Neumann problem
\begin{equation}
  \begin{cases}
  - \Delta_g u = \lambda \abs{\rd\phi}^2 u & \text{in } \Omega_\CC \\
  \del_\nu u = 0 & \text{on } \del \Omega_\CC.
\end{cases}
\end{equation}
If the boundary is regular enough,  the Neumann spectrum is discrete and we aim for a Weyl law of the form
\begin{equation}
\label{eq:netrusov}
  N_{\text{Neu}}(\lambda) = \frac{\operatorname{Area}(\Omega)}{4\pi}  \lambda + \smallo \lambda.
\end{equation}
This problem is well studied,  and,  in particular,  we already know that the Weyl
élaw holds for a large class of domains with rough boundary. For instance, it is shown in \cite{netrusovsafarov} that \eqref{eq:netrusov}  for every
domain whose boundary is of the H\"older class $\RC^{0,\alpha}$ for $\alpha >
1/2$, see also \cite{netrusov} It is also shown in \cite{netrusovsafarov} that domains with finite
straight cusps of any speed satisfy \eqref{eq:netrusov}. Moreover, sharp remainder estimates have been obtained in many cases.

A straightforward adaptation of the methods developed in this paper yields an alternative
proof  of \eqref{eq:netrusov}  provided $\abs{\rd \phi}^2 \in \RL \log \RL(\Omega_\CC)$. This
is shown in essentially the same way as Lemma \ref{lem:controlledcounting}. While this approach does not give sharp remainder estimates, it is significantly more elementary.
Using \eqref{eq:cusp}, we
note that the class of domains for which $\abs{\rd \phi}^2 \in \RL \log
\RL(\Omega_\CC)$ includes any cusp of polynomial speed.

\bibliographystyle{alpha}
\bibliography{weyl_lipnew}

\end{document}